\documentclass[12pt,oneside,reqno]{amsart}
\usepackage{graphicx}
\usepackage{caption,subcaption}
\usepackage{dcolumn}
\usepackage{amsmath,amsthm,amssymb}

\usepackage{url}
\usepackage{latexsym, graphicx,cite}

\makeatletter\@addtoreset {equation}{section}\makeatother

\usepackage[colorlinks=true, linkcolor=blue, bookmarks=true]{hyperref}

\usepackage{euscript}

\newcommand{\beq}{\begin{equation}}
\newcommand{\eeq}{\end{equation}}
\newcommand{\ber}{\begin{array}}
\newcommand{\eer}{\end{array}}

\newcommand{\ena}{\end{eqnarray}}
\newcommand{\beqa}{\begin{eqnarray}}
\newcommand{\eeqa}{\end{eqnarray}}
\newcommand{\bea}{\begin{eqnarray}}
\newcommand{\eea}{\end{eqnarray}}

\newtheorem{theorem}{Theorem}

\newtheorem{rem}{Remark}[section]
\newtheorem{lemma}{Lemma}[section]

\newcommand{\minPlusOne}[1]{
  \min\!\left(#1\right) + 1
}

\usepackage{amsthm}
\usepackage{latexsym}
\usepackage{amssymb}
\usepackage{amsmath}
\usepackage{mathrsfs}

\begin{document}

\title[Ground state of the conformal flow]{Ground state of the conformal flow on $\mathbb{S}^3$}

\author{Piotr Bizo\'n}
\address{Institute of Physics, Jagiellonian
University, Krak\'ow, Poland\newline \indent Department of Physics,  Chulalongkorn University,
Bangkok, Thailand}
\email{bizon@th.if.uj.edu.pl}

\author{Dominika Hunik-Kostyra}
\address{Institute of Physics, Jagiellonian
University, Krak\'ow, Poland}
\email{dominika.hunik@uj.edu.pl}

\author{Dmitry Pelinovsky}
\address{Department of Mathematics, McMaster University, Hamilton, Ontario  L8S 4K1, Canada}
\email{dmpeli@math.mcmaster.ca}


\date{\today}

\begin{abstract}
We consider the conformal flow model derived in \cite{bal} as a normal form for the conformally invariant cubic wave equation on $\mathbb{S}^3$. We prove that the energy attains a global constrained maximum at a family of particular stationary solutions which we call the ground state family. Using this fact and spectral properties of the linearized flow (which are interesting on their own due to a supersymmetric structure) we prove nonlinear orbital stability of the ground state family.  The main difficulty in the proof is due to the degeneracy of the ground state family as a constrained maximizer of the energy.
\end{abstract}

\maketitle

\section{Introduction}

Long-time behavior of nonlinear dispersive waves on a compact manifold can be very rich and complex because, in contrast to unbounded domains, waves cannot disperse to infinity and keep self-interacting for all times (even the global regularity of arbitrarily small solutions is a nontrivial issue). A major mathematical challenge in this context is to describe the energy transfer  between eigenmodes of the corresponding linearized flow
near the zero equilibrium. A simple  model for gaining  insight into this problem is the conformally invariant cubic wave equation on the three-sphere.
The key feature of this model is the fully resonant linearized spectrum. As a consequence, the long-time behavior of small solutions of this equation
is well approximated by solutions of an infinite-dimensional time-averaged Hamiltonian system which governs resonant interactions between the modes. This system, called the \emph{conformal flow} on $\mathbb{S}^3$, has been introduced and studied in \cite{bal}.

Among its many remarkable features (in particular, low dimensional invariant subspaces), the conformal flow has been found to admit a wealth of \emph{stationary states}, i.e. solutions for which no energy transfer between the modes occurs. In this paper we  show that among the stationary states there is a distinguished one, hereafter called the \emph{ground state}, which is a global constrained maximizer of the energy. The main body  of the paper is devoted to proving nonlinear orbital stability of the ground state.

In terms of complex Fourier coefficients
$\left(\alpha_n(t) \right)_{n \in \mathbb{N}}$, the conformal flow system takes the form (see \cite{bal} for the details of the derivation)
\beq
i(n+1) \frac{d \alpha_n}{d t} = \sum\limits_{j=0}^\infty \sum_{k=0}^{n+j} S_{njk,n+j-k}\,\bar\alpha_j \alpha_k \alpha_{n+j-k}\,,
\label{flow}
\eeq
where
\beq
 S_{njk,n+j-k} = \minPlusOne{n,j,k,n+j-k}.
\eeq
This  is  the Hamiltonian system with the conserved energy function
\begin{equation}
 H(\alpha) = \sum\limits_{n=0}^{\infty}\sum\limits_{j=0}^{\infty} \sum\limits_{k=0}^{n+j} S_{njk,n+j-k}\,\bar\alpha_n \bar \alpha_j \alpha_k \alpha_{n+j-k}
\label{Hconf}
\end{equation}
and symplectic form $\sum_n 2i(n+1)\,d\bar\alpha_n\wedge d\alpha_n$ so that
equations of motion (\ref{flow}) can be written in the form
\begin{equation}\label{ham-eq}
i (n+1) \frac{d \alpha_n}{d t} = \frac{1}{2}\,\frac{\partial H}{\partial \bar \alpha_n}\,.
\end{equation}
The conformal flow system  enjoys the following three one-parameter groups of symmetries:
 \begin{eqnarray}
 \mbox{Scaling:} \,&& \alpha_n(t) \rightarrow c \alpha_n(c^2 t),\label{symmscale}\\
 \mbox{Global phase shift:} \,&& \alpha_n(t) \rightarrow e^{i\theta} \alpha_n(t),\label{symmshift1}\\
 \mbox{Local phase shift:} \, &&\alpha_n(t) \rightarrow e^{i n \mu} \alpha_n(t),\label{symmshift2}
 \end{eqnarray}
where  $c$, $\theta$, and $\mu$ are real parameters. The latter two symmetries give rise
to two additional conserved quantities:
\begin{eqnarray}
Q(\alpha) & =&\sum\limits_{n=0}^{\infty} (n+1) |\alpha_n|^2,\label{charge}\\
E(\alpha) & =&\sum\limits_{n=0}^{\infty} (n+1)^2 |\alpha_n|^2\,. \label{lenergy}
\end{eqnarray}

\subsubsection*{Notation}
We denote the set of nonnegative integers by $\mathbb{N}$ and the set of positive integers by $\mathbb{N}_+$.
A sequence $(\alpha_n)_{n\in \mathbb{N}}$ is denoted for short by $\alpha$.
The space of square-summable sequences  is denoted by~$\ell^2$.
Given $s>0$, we define the weighted space of sequences
\begin{equation}\label{hs}
  h^s:=\left\{ \alpha \in \ell^2(\mathbb{N}): \quad \sum_{n=0}^{\infty} (n+1)^{2s} |\alpha_n|^2 <\infty \right\},
\end{equation}
endowed with its natural norm. We write $X \lesssim Y$ to denote the statement that $X\leq C Y$
for some universal  (i.e., independent of other parameters) constant $C>0$.
\vskip 0.2cm

The following theorem  states that the conformal flow
is globally well-posed in $h^1$. This is not an optimal result in terms of regularity of initial data, however it is sufficient  for our purposes. Note that
 all three conserved quantities $H$, $Q$, $E$ are well defined
in $h^1$.

\begin{theorem}
\label{theorem-evolution}
For every initial data $\alpha(0) \in h^1$, there exists a unique
global-in-time solution $\alpha(t) \in C(\mathbb{R}, h^1)$
of the system  (\ref{flow}). Moreover,
for every $t$,
$$
H(\alpha(t)) = H(\alpha(0)), \quad
Q(\alpha(t)) = Q(\alpha(0)), \quad
E(\alpha(t)) = E(\alpha(0)).
$$
\end{theorem}

A solution of \eqref{flow} is called a \emph{stationary state} if $|\alpha(t)|$ is time-independent.
A stationary state is called a \emph{standing wave} if it has the form
 \begin{equation}\label{standing}
  \alpha(t) = A e^{-i \lambda t},
\end{equation}
where the complex amplitudes $( A_n )_{n \in \mathbb{N}}$ are time-independent and the parameter $\lambda$  real.
Substituting \eqref{standing} into \eqref{flow}, we get a nonlinear system of algebraic equations
for the amplitudes:
\begin{equation}\label{sys}
 (n+1) \lambda A_n=\sum\limits_{j=0}^{\infty} \sum\limits_{k=0}^{n+j} S_{n,j,k,n+j-k} \,\bar A_j A_k A_{n+j-k}\,.
\end{equation}
This equation can be written as
$$
\lambda \dfrac{\partial Q}{\partial \bar A_n} = \dfrac{1}{2} \dfrac{\partial H}{\partial \bar A_n},
$$
hence standing waves
 \eqref{standing} admit a variational characterization as the critical points of the  functional
\begin{equation}\label{K}
K(\alpha) = \frac{1}{2} H(\alpha) - \lambda  Q(\alpha).
\end{equation}
Equivalently, standing waves \eqref{standing} are  critical points of $H$ for fixed $Q$,
where $\lambda$ is a Lagrange multiplier.

The simplest solutions of \eqref{sys} are the \emph{single-mode} states given, for any  $N \in \mathbb{N}$, by
 \begin{equation}\label{1mode}
A_n = c\, \delta_{nN}, \quad \lambda = |c|^2, \quad c\in \mathbb{C}.
 \end{equation}
In a separate paper we classify all stationary states that bifurcate from the single-mode states \cite{BHP2}. Here we are concerned exclusively with the following family of standing waves  that bifurcates from the $N=0$ single-mode state
\begin{equation}
 \label{r1}
A_n = c\, p^n, \quad \lambda = \frac{|c|^2}{(1-|p|^2)^2},
\end{equation}
where $c, p \in \mathbb{C}$ and $|p| < 1$. We shall refer to \eqref{r1} as  the \emph{ground state} family because it is the global maximizer of $H$ for fixed $Q$, as follows from our next theorem.
\begin{theorem}
\label{theorem-bound}
For every $\alpha\in h^{1/2}$ the following inequality holds
\begin{equation}
\label{bound-H}
H(\alpha) \leq  Q(\alpha)^2.
\end{equation}
Moreover, $H(\alpha) = Q(\alpha)^2$ if and only if $\alpha_n= c p^n$ for some $c,p \in \mathbb{C}$ with $|p|<1$.
\end{theorem}
The ground state family \eqref{r1} is parametrized by two complex parameters $c$ and $p$ with $|p|<1$.
In what follows, by the ground state we shall mean the normalized solution
\begin{equation}
\label{ground-state}
A_n(p) = (1-p^2) p^n, \quad \lambda = 1,
\end{equation}
parametrized by one real parameter $p \in [0,1)$. The conserved quantities $H$, $Q$, and $E$
for the ground state are
\begin{equation}\label{GS-conserved}
H(A(p))=1,\quad   Q(A(p))=1,\quad E(A(p))=\frac{1+p^2}{1-p^2}.
\end{equation}
 By the ground state \emph{orbit} (for a given $p$)  we shall mean the set obtained from
 the ground state $A(p)$ by acting on it  with the gauge  symmetries (\ref{symmshift1}) and (\ref{symmshift2}):
\begin{equation}\label{orbit}
  \mathcal{A}(p)= \{\left(e^{i\theta+i \mu n} A_n(p)
\right)_{n\in \mathbb{N}} : (\theta,\mu)\in \mathbb{S}^1 \times \mathbb{S}^1\}\,.
\end{equation}

Our main goal is to show that the ground state is \emph{orbitally stable}, i.e.
a perturbed ground state with given $p$ stays close to its orbit $\mathcal{A}(p)$
for all later times provided that the perturbation is small enough.
To measure the distance (in some norm $X$)  between the solution and the ground state orbit, we introduce the metric
\begin{equation}\label{dist}
\mathrm{dist}_X \left(\alpha(t),\mathcal{A}(p) \right):=\inf_{\theta,\mu \in \mathbb{S}}
\|\alpha(t)-e^{i \theta + i \mu \cdot} A(p(t))\|_{X}\,.
\end{equation}

We shall study orbital stability using the  Lyapunov method based on the variational formulation
\eqref{K}, spectral analysis, and coercivity estimates. The  main difficulty is due to the degeneracy of
the ground state as a constrained maximizer of energy. To eliminate this degeneracy
we shall use the symplectic orthogonal decomposition, combined with conservation laws and gauge symmetries.

For $p=0$ the ground state reduces to the single-mode state
\begin{equation}\label{deltaKronecker}
A_n(0)=\delta_{n0}, \quad \lambda = 1,
\end{equation}
and the symmetry orbit \eqref{orbit} becomes one-dimensional.
We will show that in this case a two-parameter orthogonal decomposition
involving the generators of scaling  \eqref{symmscale} and gauge \eqref{symmshift1}
symmetries suffices to eliminate the degeneracy. Having that and using
the conservation laws we establish the orbital stability of the single-mode state,
as stated in the following theorem.

\begin{theorem}
\label{theorem-one-mode}
For every small $\epsilon > 0$, there is $\delta > 0$ such that for every initial data
$\alpha(0)\in h^1$ with $\| \alpha(0) - A(0) \|_{h^1} \leq \delta$,
the corresponding unique solution $\alpha(t) \in C(\mathbb{R},h^1)$ of
(\ref{flow}) satisfies for all $t$
\begin{equation}
\label{bound-one-mode}
\mathrm{dist}_{h^1} (\alpha(t), \mathcal{A}(0)) \leq \epsilon.
\end{equation}
\end{theorem}

For $p \in (0,1)$, we need to introduce the four-parameter orthogonal decomposition involving both
gauge symmetries (\ref{symmshift1}) and (\ref{symmshift2}), the scaling symmetry (\ref{symmscale}),
and the parameter $p$ itself. Having that and using the variational characterization
of the constrained maximizers in (\ref{K}), together with the conservation of $E$,
we are able to partially eliminate  the degeneracy  in the following sense:
\begin{itemize}
\item[(i)] the distance between the solution $\alpha(t)$ and the ground state orbit
$\mathcal{A}(p(t))$ is bounded in the norm $h^{1/2}$;
\item[(ii)] the parameter $p(t)$ may drift in time
towards smaller values compensated by the increasing $h^1$ distance
between the solution $\alpha(t)$ and the orbit $\mathcal{A}(p(t))$.
\end{itemize}
More precisely, we have the following
nonlinear orbital stability result for the ground state:

\begin{theorem}
\label{theorem-ground-state}
For every $p_0 \in (0,1)$ and every small $\epsilon > 0$,
there is $\delta > 0$ such that for every initial data
$\alpha(0) \in h^1$ satisfying $\| \alpha(0) - A(p_0) \|_{h^1} \leq \delta$,
 the corresponding unique solution $\alpha(t) \in C(\mathbb{R}_+,h^1)$
of (\ref{flow}) satisfies for all $t$
\begin{equation}
\label{bound-ground-state}
\mathrm{dist}_{h^{1/2}} \left(\alpha(t) - \mathcal{A}(p(t))\right) \leq \epsilon,
\end{equation}
and
\begin{equation}
\label{bound-ground-state-bigger}
 \mathrm{dist}_{h^{1}} \left(\alpha(t) - \mathcal{A}(p(t)) \right) \lesssim \epsilon + (p_0 - p(t))^{1/2}
\end{equation}
for some continuous function $p(t) \in [0,p_0]$.
\end{theorem}

\begin{rem}
The drift term in \eqref{bound-ground-state-bigger} is, of course, undesired for the orbital stability
of each point in the ground state family. It is an interesting open problem whether it reflects the
technical limitation of our approach or, indeed, the drift towards the orbit $\mathcal{A}(p)$
with a smaller $p(t) \in [0,p_0)$ can actually occur.
\end{rem}

Although it is not needed for the orbital stability, we shall give an explicit description of the spectrum of the linearized operator around the ground state. We believe that the results of this spectral analysis are interesting on their own. In particular, some intriguing algebraic properties of the spectrum may provide a hint in searching for a Lax pair for the conformal flow.

The conformal flow system \eqref{flow} is structurally similar to the Fourier representations
of the cubic Szeg\H{o} equation \cite{GG10,GG12,GG15} and the Lowest Landau Level (LLL) equation \cite{BBCE,Hani2,Hani3}.
These two  equations also possess ground states that saturate inequalities analogous to \eqref{bound-H}.
Their nonlinear orbital stabilities were established in \cite{GG10} and \cite{Hani1,Hani2}, respectively,
by compactness-type arguments. Such arguments are shorter compared to the Lyapunov approach, however
we prefer the latter because it gives a hands-on control of  the perturbations and, more importantly,
can also be  applied to \emph{local} constrained minimizers (or maximizers). Having said this,
we shall also sketch an alternative proof of Theorem~4, following very closely the proof
of orbital stability of the ground state for the cubic Szeg\H{o} equation given in \cite{GG10}.
As we shall see, both approaches lead to the essentially equivalent conclusions, in particular
the compactness approach also does not eliminate the possibility of drift of $p(t)$.

The reason that such drift is not present for perturbations of the ground state of
the cubic Szeg\H{o} equation is due to the fact that this ground state  is a maximizer
of energy under \emph{two} constraints (rather than one, as in our case). On the other hand,
for the ground state of the LLL equation, which is the maximizer of energy under one
constraint (as in our case), the drift of the parameter analogous to $p$  is included into consideration
by an action of the symmetry transformation (namely, the magnetic translation). Unfortunately,
no analogue of the additional symmetry transformation exists for the conformal flow, as far as we can see.

The paper is organized as follows. Theorems \ref{theorem-evolution} and  \ref{theorem-bound}  are proved in
sections \ref{section-Cauchy} and \ref{section-characterization}, respectively.
Section \ref{section-stability} introduces the second variation
of the energy function and the spectral stability problem.
Sections \ref{section-one-mode} and \ref{section-geometric-series} are devoted to the
spectral stability analysis of the single-mode states and the ground state.
Theorems  \ref{theorem-one-mode} and \ref{theorem-ground-state} are proved
in sections \ref{section-nonlinear} and \ref{section-ground-state}, respectively.
Appendix A gives some summation identities used in the paper.

\section{Global solutions to the Cauchy problem}
\label{section-Cauchy}

The global well-posedness of the conformal flow for $h^1$ initial data,
stated in Theorem~\ref{theorem-evolution}, is a simple corollary of the following local well-posedness result
and the conservation of $E$.
\begin{lemma}
\label{lemma-evolution}
For every initial data $\alpha(0) \in h^1$, there exists time $T > 0$ and a unique
solution $\alpha(t) \in C((-T,T), h^1)$ of (\ref{flow}). Moreover, the map $\alpha(0)
\mapsto \alpha(t)$ is continuous in $h^1$ for every $t \in (-T,T)$.
\end{lemma}

\begin{proof}
Let us rewrite the  system (\ref{flow}) in the integral form
\begin{equation}
\label{Duhamel-form}
\alpha(t) = \alpha(0) - i \int_0^t F(\alpha(\tau)) d \tau,
\end{equation}
where the vector field $F$ is given by
\begin{equation}
\label{vector-field}
[F(\alpha)]_n := \sum\limits_{j=0}^\infty \sum_{k=0}^{n+j} \frac{\min(n,j,k,n+j-k)+1}{n+1}
\,\bar\alpha_j \alpha_k \alpha_{n+j-k}\,.
\end{equation}
Let us show that
\begin{equation}\label{F-bound}
\| F(\alpha) \|_{h^1} \lesssim \|\alpha \|^3_{h^1}.
\end{equation}
Indeed, by using the Fourier transform
$$
u(\theta) = \sum_{n \in \mathbb{N}} |\alpha_n| e^{i n \theta}, \quad \theta \in \mathbb{S}
$$
and the inequality
$$
\frac{\min(n,j,k,n+j-k)+1}{n+1} \leq 1,
$$
we obtain
$$
\| F(\alpha) \|_{h^1} \leq \left\| \sum_{j=0}^\infty \sum_{k=0}^{n+j} |\alpha_j| |\alpha_k| |\alpha_{n+j-k}|  \right\|_{h^1}
\lesssim \| u^3 \|_{H^1(\mathbb{S})} \lesssim \| u \|^3_{H^1(\mathbb{S})} \lesssim \| \alpha \|^3_{h^1},
$$
where we have used the fact that $H^1(\mathbb{S})$ is a Banach algebra with respect to pointwise multiplication.
From \eqref{F-bound} it follows by  the Picard method based on the fixed-point argument on a small time interval
$(-T,T)$ that there exists a local solution to the integral equation (\ref{Duhamel-form}) for
every initial data $\alpha(0) \in h^1$  and the mapping $\alpha(0)\mapsto \alpha(t)$ is continuous.
The time $T$ is inversely proportional to $\| \alpha(0) \|_{h^1}^2$.
\end{proof}

Conservation of $H$, $Q$, and $E$ follow from the symmetries of the conformal flow system (\ref{flow}).
Since the squared norm $\|\alpha(t)\|^2_{h^1}=E(\alpha(t))$ is conserved in time,
the lifespan $T$ in Lemma \ref{lemma-evolution} can be extended to infinity, which concludes the proof of Theorem~1.

\begin{rem}
The  global well-posedness result in Theorem~1 can be extended to the spaces $h^s$ with $s>1/2$
and quite possibly even to the critical
space $h^{1/2}$, in analogy to the global well-posedness result for the cubic Szeg\H{o} equation
(see Theorem 2.1 in \cite{GG10}). However,
the result of Theorem \ref{theorem-evolution} is sufficient for our purposes because the orbital
stability analysis relies on the global solution in $h^1$ only.
\end{rem}

\section{Global energy bound}
\label{section-characterization}

The following lemma establishes the inequality (\ref{bound-H})  and shows that it is saturated only for the geometric sequences. Theorem \ref{theorem-bound} is the reinstatement of the same result.

\begin{lemma}
\label{lem-bound}
For every complex-valued sequence $A \in h^{1/2}$ the following inequality holds
\begin{equation}
\label{H-inequality}
\sum\limits_{n=0}^{\infty}\sum\limits_{j=0}^{\infty} \sum\limits_{k=0}^{n+j} \left[\min(n,j,k,n+j-k)+1\right] \,\bar{A}_n \bar{A}_j A_k A_{n+j-k} \leq
\left[ \sum_{n=0}^{\infty} (n+1) |A_n|^2 \right]^2.
\end{equation}
Moreover, this inequality  is saturated if and only if  $A_n = c p^n$ for $c,p\in  \mathbb{C}$ with $|p|<1$.
\end{lemma}

\begin{proof}
The left-hand side of (\ref{H-inequality}) can be written as
\begin{equation}\label{sl}
\sum\limits_{n=0}^{\infty}\sum\limits_{j=0}^{n} \sum\limits_{k=0}^{n} \left[\min(j,n-j,k,n-k)+1\right] \,\bar{A}_j \bar{A}_{n-j} A_k A_{n-k},
\end{equation}
whereas the right-hand side of (\ref{H-inequality}) can be written as
\begin{equation}\label{sr}
\sum_{n=0}^{\infty} \sum_{k=0}^{n} (k+1) (n+1-k) |A_k|^2 |A_{n-k}|^2.
\end{equation}
Let us fix $n \in \mathbb{N}$ and denote $x_k := A_k A_{n-k}$ for $0 \leq k \leq n$.
To prove the inequality \eqref{H-inequality} it suffices to show that the following  quadratic form is nonnegative
$$
\sum_{k=0}^{n} (k+1) (n+1-k) |x_k|^2 -
\sum\limits_{j=0}^{n} \sum\limits_{k=0}^{n} \left[\min(j,n-j,k,n-k)+1\right] \overline{x}_j x_k.
$$
Let us prove this for odd $n = 2N + 1$ with $N \in \mathbb{N}$. The proof for even $n$ is analogous.
Since $x_k = x_{n-k}$ by definition, we can simplify the sums \eqref{sl} and \eqref{sr}  as follows:
\begin{equation}\label{s2}
\sum\limits_{j=0}^{n} \sum\limits_{k=0}^{n} \left[\min(j,n-j,k,n-k)+1\right] \overline{x}_j x_k =
4 \sum\limits_{j=0}^{N} \sum\limits_{k=0}^{N} \left[\min(j,k)+1 \right] \overline{x}_j x_k.
\end{equation}
and
\begin{equation}\label{s1}
\sum_{k=0}^{n} (k+1) (n+1-k) |x_k|^2 = 2 \sum_{k=0}^{N} (k+1) (2N+2-k) |x_k|^2.
\end{equation}
Subtracting \eqref{s2} from \eqref{s1} we obtain the following identity
\begin{eqnarray}\label{min}
& 2 \sum_{k=0}^{N} (k+1) (2N+2-k) |x_k|^2 -
4 \sum\limits_{j=0}^{N} \sum\limits_{k=0}^{N} \left[\min(j,k)+1 \right] \overline{x}_j x_k \nonumber\\
& = 4 \sum_{j=0}^{N-1} \sum_{k=j+1}^N (j+1) |x_j - x_k |^2 \geq 0. \label{identity}
\end{eqnarray}
The identity (\ref{identity}) is proven by induction in $N$. The case $N = 0$ is trivial.
The case $N=1$ is verified by inspection. For general $N$, the difference
between the left hand sides of \eqref{min} evaluated at $N+1$ and $N$ is
\begin{eqnarray*}
& \phantom{t} & 4\sum\limits_{k=0}^N (k+1) \left(|x_k|^2 - \bar x_{N+1} x_k - x_{N+1} \bar x_k\right) + 2 (N+1)(N+2) |x_{N+1}|^2 \\
& = & 4 \sum_{k=0}^N (k+1) |x_k - x_{N+1}|^2,
\end{eqnarray*}
which is equal to the difference between the right hand sides of \eqref{min} evaluated at $N+1$ and $N$.
By induction, the identity (\ref{identity}), which holds for $N = 0,1$, will hold for any $N \in \mathbb{N}$.

Combining \eqref{min} with a similar result for $n = 2N$ and summing up with respect to $N \in \mathbb{N}$,
we obtain the inequality (\ref{H-inequality}). The inequality is saturated when the double sum in \eqref{min} vanishes which happens
 if $x_k = A_k A_{n-k}$ is independent of $k$ for every $0 \leq k \leq n$ (but may depend on
$n$ for $n \in \mathbb{N}$). This is true if and
only if $A_k = c p^k$ for some $c,p \in \mathbb{C}$.
\end{proof}

\begin{rem}
Theorem~2 provides an alternative variational characterization of the ground state family (\ref{r1}).
Indeed, from Theorem \ref{theorem-bound}, we know that $G := Q^2-H$ attains
the global minimum equal to zero at the geometric sequence $A_n=c p^n$, hence its first variation at $A$ vanishes
\begin{equation}\label{delta-G}
   G'(A) = 2 Q(A) Q'(A) - H'(A)=0.
\end{equation}
Comparing this with the variational characterization \eqref{K} we see that $A_n = c p^n$ is
a solution of (\ref{sys}) with $\lambda = Q(A) = |c|^2/(1-|p|^2)^2$.
\end{rem}

\section{Second variation and the spectral stability}
\label{section-stability}
Here we compute the second variation of the functional $K$, defined in \eqref{K}, at a general stationary solution  and use this result to formulate the spectral stability problem.

Let $\alpha = A + a + i b$, where $A$ is a real root of the algebraic system \eqref{sys}, whereas $a$
and $b$ are real and imaginary parts of the perturbation.
Because the stationary solution $A$ is a critical point of $K$, the first variation of $K$
vanishes at $\alpha = A$ and the second variation of $K$ at $\alpha = A$
can be written as a quadratic form associated with the Hessian
operator. In variables above, we obtain the quadratic form in the diagonalized form:
\beq
\label{expansions-K}
K(A + a + i b) - K(A) = \langle L_+ a, a \rangle + \langle L_- b, b \rangle + \mathcal{O}(\| a \|^3 + \| b \|^3),
\eeq
where $\langle \cdot, \cdot \rangle$ is the inner product in $\ell^2(\mathbb{N})$ and $\| \cdot \|$ is the induced norm.

After straightforward computations, we obtain the explicit form for the self-adjoint operators
$L_{\pm} : D(L_{\pm}) \to \ell^2$, where
$D(L_{\pm}) \subset \ell^2$ is the maximal domain of the unbounded operators $L_{\pm}$.
After the interchange of summations, we obtain
\begin{eqnarray}
\nonumber
(L_{\pm} a)_n & = & \sum_{j = 0}^{\infty} \sum_{k = 0}^{n+j} S_{njk,n+j-k} \left[ 2 A_j A_{n+j-k} a_k \pm A_k A_{n+j-k} a_j \right]
- (n+1) \lambda a_n \\
& = & (B_{\pm} a)_n - (n+1) \lambda a_n,
\label{Hessian}
\end{eqnarray}
where
$$
(B_{\pm} a)_n = \sum_{j = 0}^{\infty} \left[ 2\sum_{k = \max(0,j-n)}^\infty S_{njk,n+k-j} A_k A_{n+k-j}
\pm \sum_{k=0}^{n+j} S_{njk,n+j-k} A_k A_{n+j-k} \right] a_j.
$$

The Hessian operator given by the self-adjoint operators $L_{\pm}$ also defines the linearized
stability problem for the small perturbations to the stationary states.
Let us consider the following decomposition of solutions of the conformal flow (\ref{flow}),
\begin{equation}
\label{decomposition}
\alpha(t) = e^{-i \lambda t} \left[ A + a(t) + i b(t) \right],
\end{equation}
with real $A$, $a$ and $b$. When the nonlinear system  (\ref{flow}) is truncated at the linearized approximation
with respect to $a$ and $b$, we obtain the linearized evolution system
\beq
\label{evolution-linear}
M \frac{da}{dt} = L_- b, \quad M \frac{db}{dt} = -L_+ a,
\eeq
where $M = {\rm diag}(1,2,...)$ is the diagonal matrix operator of positive integers.
Substituting $a(t) = e^{\Lambda t} \texttt{a}$ and
$b(t) = e^{\Lambda t} \texttt{b}$ into \eqref{evolution-linear} we get the eigenvalue problem
\beq
\label{spectrum}
L_- \texttt{a} = \Lambda M \texttt{b}, \quad L_+ \texttt{b} = -\Lambda M \texttt{a}.
\eeq
We say that the stationary solution $A$ is {\em spectrally stable}
if all the eigenvalues $\Lambda$ lie on the imaginary axis.

\section{Spectral stability of the single-mode states}
\label{section-one-mode}

Here we compute $L_{\pm}$ and eigenvalues of the spectral stability
problem (\ref{spectrum}) for the single-mode states (\ref{1mode}).
Because of the scaling transformation (\ref{symmscale}), we set $c = 1$
and hence $\lambda = 1$. In this case, the self-adjoint operators
$L_{\pm} : D(L_{\pm}) \to \ell^2(\mathbb{N})$ are given explicitly by
\beq
\label{Hessian-one-mode}
(L_{\pm} a)_n = \left[ 2 \min(n,N) + 1 - n\right] a_n \pm \left[ \min(n,N,2N-n) + 1\right] a_{2N-n}, \quad n \in \mathbb{N},
\eeq
from which it is clear that $D(L_{\pm}) = h^{1}$.

We prove the following characterization of eigenvalues of $L_{\pm}$,
from which we obtain the spectral stability of the single-mode states (\ref{1mode}).

\begin{lemma}
\label{lemma1}
The single-mode state (\ref{1mode}) with $N \in \mathbb{N}$
is a degenerate saddle point of $K$ with $2N+1$ positive eigenvalues (counted with multiplicities),
the zero eigenvalue of multiplicity $2N+3$ and infinitely many negative eigenvalues bounded away from zero.
\end{lemma}

\begin{proof}
The operators $L_{\pm}$ in (\ref{Hessian-one-mode}) consist of a $(2N+1)$-$(2N+1)$ block denoted by $\tilde{L}_{\pm}$
and a diagonal block with entries $\{2N+1-n\}_{n \geq 2N+1}$. The latter diagonal block
has one zero eigenvalue and all other eigenvalues are strictly negative. The former block
can be written in the form
$$
\tilde{L}_+ = \left[ \begin{array}{ccc} L_{11} & 0 & L_{12} \\ 0 & 2(N+1) & 0 \\ L_{12}^T & 0 & L_{22} \end{array} \right], \quad
\tilde{L}_- = \left[ \begin{array}{ccc} L_{11} & 0 & -L_{12} \\ 0 & 0 & 0 \\ -L_{12}^T & 0 & L_{22} \end{array} \right],
$$
where $L_{11} = {\rm diag}(1,2,...,N)$, $L_{12} = {\rm antidiag}(1,2,...,N)$, and $L_{22} = {\rm diag}(N,N-1,...,1)$.
Since the eigenvalue problem for $\tilde{L}_{\pm}$ decouples into $N$ pairs and one equation, we can count the
eigenvalues of $\tilde{L}_{\pm}$. The block $\tilde{L}_+$ has $N+1$ positive eigenvalues and the zero
eigenvalue of multiplicity $N$. The block $\tilde{L}_-$ has $N$ positive eigenvalues and the zero
eigenvalue of multiplicity $N + 1$. The assertion of the lemma follows by combining the count of all eigenvalues.
\end{proof}

\begin{lemma}
\label{lemma2}
The single-mode state (\ref{1mode}) with $N \in \mathbb{N}_+$ module to the gauge symmetry (\ref{symmshift1})
is a degenerate saddle point of $H$ under fixed $Q$
with $2N$ positive eigenvalues (counted with multiplicities),
the zero eigenvalue of multiplicity $2N+2$ and infinitely many negative eigenvalues bounded away from zero.
The single-mode state (\ref{1mode}) with $N = 0$ module to the gauge symmetry (\ref{symmshift1})
is a degenerate maximizer of $H$ under fixed $Q$ with a double zero eigenvalue.
\end{lemma}

\begin{proof}
The count of eigenvalues in Lemma \ref{lemma1} is modified by two constraints as follows.
The positive eigenvalue corresponds to the central entry $2(N+1)$ in $\tilde{L}_+$.
This positive eigenvalue is removed by the constraint of fixed $Q(\alpha)$ in the variational
formulation (\ref{K}). Indeed, if $\alpha = A + a + i b$ and $Q(\alpha)$ is fixed, we
impose the linear constraint
on the real perturbation $a$ in the form
\begin{equation}
\label{constraint-one-mode-plus}
\langle MA, a \rangle = \sum\limits_{j=0}^\infty (j+1) A_j a_j = 0,
\end{equation}
which yields the constraint $a_N = 0$ for the single-mode state (\ref{1mode}). This constraint
 removes the corresponding positive entry of $\tilde{L}_+$.

Similarly, the zero eigenvalue from the central zero entry in $\tilde{L}_-$ corresponds to
the gauge symmetry (\ref{symmshift1}). In order to define uniquely
the parameter $\theta$ due to the gauge symmetry (\ref{symmshift1}), we impose
the linear constraint on the perturbation $b$ in the form
\begin{equation}
\label{constraint-one-mode-minus}
\langle A, b \rangle = \sum\limits_{j=0}^\infty A_j b_j = 0,
\end{equation}
which yields the constraint $b_N = 0$ for the single-mode state (\ref{1mode}).
This constraint removes the corresponding zero entry of $\tilde{L}_-$.
\end{proof}

\begin{lemma}
\label{lemma3}
All single-mode states  (\ref{1mode}) are spectrally stable.
\end{lemma}

\begin{proof}
Due to the block-diagonalization of $L_{\pm}$, one can solve the spectral
stability problem (\ref{spectrum}) explicitly. Associated
to the diagonal blocks of $L_{\pm}$, we obtain an infinite sequence
of eigenvalues $\Lambda_n=\pm i \Omega_n$, where
\begin{equation}
\label{eig-negative}
\Omega_n = \frac{n-2N-1}{n+1}, \quad n \geq 2N+1.
\end{equation}
Note that $\Omega_{2N+1} = 0$ corresponds to the zero eigenvalue of $L_{\pm}$ associated with this entry,
whereas $\Omega_n > 0$ for $n \geq 2N+2$ correspond to the negative eigenvalues of $L_{\pm}$ in the corresponding
entries.

For the $(2N+1)$-$(2N+1)$ block of $L_{\pm}$ denoted by $\tilde{L}_{\pm}$, we
obtain a sequence of $N$ eigenvalues $\Lambda_n =\pm i \Omega_n$, where
\begin{equation}
\label{eig-positive}
\Omega_n = \frac{2(N-n)}{2N+1-n}, \quad 0 \leq n \leq N-1.
\end{equation}
These eigenvalues correspond to the positive eigenvalues of $L_{\pm}$. In addition,
we count the zero eigenvalue $\Omega = 0$ of geometric multiplicity $2N+1$ and algebraic multiplicity $2N+2$
in the spectral problem (\ref{spectrum}) associated with the same block $\tilde{L}_{\pm}$.
No other eigenvalues exist, hence for every $N \in \mathbb{N}$ the single-mode state (\ref{1mode})
  is spectrally stable.
\end{proof}

\begin{rem}
All eigenvalues of the spectral problem (\ref{spectrum})  for the single-mode state (\ref{1mode})
with any $N \in \mathbb{N}$ are semi-simple, except for the double
zero eigenvalue related to the gauge symmetry (\ref{symmshift1}).
\end{rem}

\section{Spectral stability of the ground state}
\label{section-geometric-series}

In the case of the ground state \eqref{ground-state}, the self-adjoint operators $L_{\pm} : D(L_{\pm}) \to \ell^2$
defined by (\ref{Hessian}) take the following  form
\begin{eqnarray}
\label{Hessian-r2}
[L_{\pm}(p) a]_n = \sum_{j=0}^{\infty} [B_{\pm}(p)]_{nj} a_j - (n+1)a_n,
\end{eqnarray}
where
\begin{eqnarray}
\nonumber
[B_{\pm}(p)]_{nj} & = & (1-p^2)^2 \left[ 2\sum_{k = \max(0,j-n)}^\infty S_{njk,n+k-j} p^{n+2k-j} \pm
\sum_{k=0}^{n+j} S_{njk,n+j-k} p^{n+j} \right] \\
\label{Hessian-explicit}
& = & 2 p^{|n-j|}- 2 p^{2+n+j} \pm (1-p^2)^2 (j+1)(n+1) p^{n+j}.
\end{eqnarray}
To derive this expression we have used relations~\eqref{appendix1}-\eqref{appendix3} from the appendix.
Note that the first term in $B_{\pm}(p)$ is given by the Toeplitz operator,
while the second and third terms in $B_{\pm}(p)$ are given by outer products.

First, we establish commutativity of the linear operators given by (\ref{Hessian-r2})--(\ref{Hessian-explicit}).

\begin{lemma}
\label{lemma-commutativity}
For every $p \in [0,1)$, we have
\begin{equation}
\label{commutation}
[L_+(p), L_-(p)] = 0
\end{equation}
and
\begin{equation}
\label{commutation-tilde}
[M^{-1} L_+(p), M^{-1} L_-(p)] = 0.
\end{equation}
\end{lemma}

\begin{proof}
In order to verify the commutation relation \eqref{commutation}, we use (\ref{Hessian-r2})--(\ref{Hessian-explicit}) and write
\begin{eqnarray}
\label{appendix5}
\left[ L_+(p) L_-(p) - L_-(p) L_+(p) \right]_{nj} = 2 (1-p^2)^2 C_{nj},
\end{eqnarray}
where
\begin{eqnarray*}
C_{nj} & = &  \sum_{k=0}^{\infty} (1+k)(1+n) p^{n+k}
\left[ 2 p^{|k-j|} - 2 p^{k+j+2} - (k+1) \delta_{kj} \right] \\
& \phantom{t} & \quad \quad \quad - (1+k)(1+j) p^{k+j}
\left[ 2 p^{|k-n|} - 2 p^{k+n+2} - (k+1) \delta_{kn} \right].
\end{eqnarray*}
To show that $C_{nj} = 0$, we proceed for $j \geq n$ (the proof for $j \leq n$ is analogous):
\begin{eqnarray*}
C_{nj} & = & 2 \sum_{k=0}^{\infty} (1+k)(1+n)p^{n+k+|k-j|} - 2 \sum_{k=0}^{\infty} (1+k)(1+j) p^{j+k+|k-n|} \\
& \phantom{t} & + 2 \sum_{k=0}^{\infty} (1+k) (j-n) p^{n+j+2k+2} + (1+n)(1+j)(n-j) p^{n+j}
\end{eqnarray*}
such that
\begin{eqnarray*}
C_{nj} & = & 2 \sum_{k=0}^{j} (1+k)(1+n)p^{n+j} + 2 \sum_{k=j+1}^{\infty} (1+k)(1+n) p^{n+2k-j} \\
& \phantom{t} & - 2 \sum_{k=0}^{\infty} (1+k)(1+j) p^{n+j} - 2 \sum_{k=n+1}^{\infty} (1+k)(1+j) p^{j+2k-n} \\
& \phantom{t} & + 2 \sum_{k=0}^{\infty} (1+k) (j-n) p^{n+j+2k+2} + (1+n)(1+j)(n-j) p^{n+j}.
\end{eqnarray*}
This yields $C_{nj} = 0$ in view of identities (\ref{appendix7})--(\ref{appendix8}) from the appendix.

In order to verify the commutation relation \eqref{commutation-tilde}, we use (\ref{Hessian-r2})--(\ref{Hessian-explicit}) and write
\begin{eqnarray}
\label{appendix11}
\left[ L_+(p) M^{-1} L_-(p) - L_-(p) M^{-1} L_+(p) \right]_{nj} = 2 (1-p^2)^2 D_{nj},
\end{eqnarray}
where
\begin{eqnarray*}
D_{nj} & = &  \sum_{k=0}^{\infty} (1+n) p^{n+k}
\left[ 2 p^{|k-j|} - 2 p^{k+j+2} - (k+1) \delta_{kj} \right] \\
& \phantom{t} & \quad \quad \quad - (1+j) p^{k+j}
\left[ 2 p^{|k-n|} - 2 p^{k+n+2} - (k+1) \delta_{kn} \right].
\end{eqnarray*}
As previously, we proceed for $j \geq n$ (the proof for $j \leq n$ is analogous):
\begin{eqnarray*}
D_{nj} & = & 2 \sum_{k=0}^{\infty} (1+n)p^{n+k+|k-j|} - 2 \sum_{k=0}^{\infty} (1+j) p^{j+k+|k-n|} + 2 \sum_{k=0}^{\infty} (j-n) p^{n+j+2k+2} \\
& = & 2 \sum_{k=j+1}^{\infty} (1+n) p^{n+2k-j} - 2 \sum_{k=n+1}^{\infty} (1+j) p^{j+2k-n} + 2 \sum_{k=0}^{\infty} (j-n) p^{n+j+2k+2}.
\end{eqnarray*}
This yields $D_{nj} = 0$ in view of the identity (\ref{appendix7}) from the appendix.
\end{proof}

Because of the commutation relation (\ref{commutation}), the operators $L_{\pm}(p)$ have a common basis of
eigenvectors. The following lemma fully characterizes the spectra of $L_{\pm}(p)$ in $\ell^2$.

\begin{lemma}
\label{lemma-L-2}
For every $p \in [0,1)$,  the spectra of the operators $L_{\pm}(p) : h^1 \subset \ell^2 \to \ell^2$
given by (\ref{Hessian-r2}) consist of the following isolated eigenvalues:
\begin{equation}
\label{spectrum-L-minus}
\hspace{-0.9cm} \sigma(L_-) = \{\dots,-3,-2,-1, 0, 0\},
\end{equation}
and
\begin{equation}
\label{spectrum-L-plus}
\sigma(L_+) = \{\dots,-3,-2,-1,0, \lambda_*(p)\}\,,
\end{equation}
where $\lambda_*(p) = 2(1+p^2)/(1-p^2) > 0$.
\end{lemma}

\begin{proof}
For every $p \in [0,1)$ we have $A(p) \in h^m$ for every $m \in \mathbb{N}$.
Hence $B_{\pm}(p)$ are bounded operators from $\ell^2$ to $\ell^2$,
whereas the diagonal parts of $L_{\pm}(p)$ are unbounded operators from $\ell^2$
to $\ell^2$ with the domain $h^1$. Since
the diagonal parts of $L_{\pm}(p)$ (of the Hilbert--Schmidt type) have compact resolvent
and the $B_{\pm}(p)$ operators are bounded perturbations to the diagonal parts,
the operators $L_{\pm}(p)$ have compact resolvent. Hence, the spectra of $L_{\pm}(p)$
consist of  infinitely many isolated eigenvalues.

By the symmetries (\ref{symmshift1}) and (\ref{symmshift2}), we have
\begin{equation}\label{zeroLm}
\quad L_-(p) A(p) = 0,\quad L_-(p) M A(p) = 0,
\end{equation}
where $M = {\rm diag}(1,2,\dots)$.
On the other hand, by differentiating (\ref{ground-state}) with respect to $p$, we obtain
\begin{equation}\label{zeroLp}
  L_+ A'(p) = 0, \quad \mbox{\rm where} \;\; A'(p) = - \frac{(1+p^2)}{p(1-p^2)} A(p) + \frac{1}{p} M A(p).
\end{equation}
By differentiating the scaling symmetry (\ref{symmscale}) with respect to $c$ at $c = 1$,
we get the explicit solution $a = A(p)$, $b = -2t A(p)$ of the linear equation (\ref{evolution-linear}), hence
\begin{equation}
\label{doublezeroLp}
L_+(p) A(p) = 2 M A(p).
\end{equation}
Combined with (\ref{zeroLp}), this yields
\begin{equation}
\label{pos-eig-Lp}
L_+(p) M A(p) = \frac{1+p^2}{1-p^2} L_+(p) A(p) = \frac{2(1+p^2)}{1-p^2} M A(p),
\end{equation}
which gives the positive eigenvalue $\lambda_*(p)$ in (\ref{spectrum-L-plus}).

It remains to prove that the rest of the spectrum of $L_{\pm}(p)$ coincides
with the negative integers. Since the same two-dimensional subspace $X_0(p) = {\rm span}\{ A(p), MA(p) \}$
is associated with the double zero eigenvalue of operator $L_-(p)$ and with the two simple nonnegative eigenvalues
of operator $L_+(p)$, we introduce the orthogonal complement
\begin{equation}
\label{X-0-constraints}
[X_0(p)]^{\perp} := \left\{ a \in \ell^2 : \quad \langle A(p), a \rangle = \langle M A(p), a \rangle = 0 \right\}.
\end{equation}
Eigenvectors for negative eigenvalues of $L_{\pm}(p)$ belong to $[X_0(p)]^{\perp}$.
Due to the second orthogonality condition in (\ref{X-0-constraints}), we compute
\begin{equation}
\label{notations-T}
[B_{\pm}(p) a]_{n} = \sum_{j\in \mathbb{N}} \left[ 2 p^{|n-j|}- 2 p^{2+n+j} \pm (1-p^2)^2 (j+1)(n+1) p^{n+j} \right] a_j =
2 [T(p) a]_n,
\end{equation}
where $[T(p)]_{nj} := p^{|n-j|} - p^{n+j+2}$. Hence, the negative eigenvalues of $L_{\pm}(p)$ are identical to
the negative eigenvalues of $2T(p) - M$ and hence they are identical to each other.

In order to prove that the negative eigenvalues of $2T(p) - M$ are negative integers,
let us define the shift operator $S : \ell^2 \to \ell^2$ by
$$
S: (a_0,a_1,a_2,\dots) \mapsto (0,a_0,a_1,a_2,\dots),
$$
and its left inverse operator $S^* : \ell^2 \to \ell^2$ by
$$
S^*: (a_0,a_1,a_2,\dots) \mapsto (a_1,a_2,a_3,\dots).
$$
Let us show that for every $a \in [X_0(p)]^{\perp}$,
\begin{equation}\label{LS}
  [2 T(p) - M, S] a = - S a, \qquad [2T(p) - M, S^*] a =  S^* a.
\end{equation}
Indeed, the first identity in (\ref{LS}) is verified if
the following two expressions are equal to each other:
\begin{eqnarray*}
[(2T(p)-M)Sa]_n & = & 2 \sum_{j=1}^{\infty} (p^{|n-j|} - p^{2+n+j}) a_{j-1} - (n+1) a_{n-1} \\
& = & 2 \sum_{k=0}^{\infty} (p^{|n-1-k|} - p^{3+n+k}) a_{k} - (n+1) a_{n-1}
\end{eqnarray*}
and
\begin{eqnarray*}
[S(2T(p)-M -I)a]_n & = & [(2T(p)-M -I)a]_{n-1} \\
& = & 2 \sum_{k=0}^{\infty} (p^{|n-1-k|} - p^{1+n+k}) a_{k} - (n+1) a_{n-1}.
\end{eqnarray*}
Since $\sum_{k=0}^{\infty} p^k a_k = 0$ thanks to the first orthogonality condition in (\ref{X-0-constraints}),
the two expressions are equal to each other so that the first identity in (\ref{LS}) is verified.
Now, the second identity in (\ref{LS}) is verified if
the following two expressions are equal to each other:
\begin{eqnarray*}
[(2T(p)-M)S^* a]_n & = & 2 \sum_{j=0}^{\infty} (p^{|n-j|} - p^{2+n+j}) a_{j+1} - (n+1) a_{n+1} \\
& = & 2 \sum_{k=1}^{\infty} (p^{|n+1-k|} - p^{1+n+k}) a_{k} - (n+1) a_{n+1} \\
& = & 2 \sum_{k=0}^{\infty} (p^{|n+1-k|} - p^{1+n+k}) a_{k} - (n+1) a_{n+1}
\end{eqnarray*}
and
\begin{eqnarray*}
[S^*(2T(p)-M +I)a]_n & = & [(2T(p)-M +I)a]_{n+1} \\
& = & 2 \sum_{k=0}^{\infty} (p^{|n+1-k|} - p^{3+n+k}) a_{k} - (n+1) a_{n+1}.
\end{eqnarray*}
The two expressions are equal to each other again
thanks to the first orthogonality condition in (\ref{X-0-constraints}),
so that the second identity in (\ref{LS}) is verified.

The two identities in (\ref{LS}) imply that
the operators $S$ and $S^*$ play the role of creation and annihilation operators
for elements of $[X_0(p)]^{\perp}$. In particular, they generate
the (same) set of eigenvectors of $L_{\pm}(p)$ for the (same) eigenvalues of $L_{\pm}(p)$.
More precisely, the second equation in (\ref{LS}) shows that
the negative eigenvalues of $2T(p) - M$ are located at
negative integers, $\lambda_m = -m$, $m \in \mathbb{N}_+$, whereas
the corresponding eigenvectors $v^{(m)}$, defined by
$$
L_{\pm}(p) v^{(m)} = (2T(p) - M) v^{(m)} = \lambda_m v^{(m)}, \quad m \in \mathbb{N}_+,
$$
are related by
$$
S^* v^{(m)}(p) = v^{(m-1)}(p), \quad m \geq 2
$$
and
$$
S^* v^{(1)}(p) \in {\rm ker}(2T(p)-M) = {\rm ker}(L_+(p)) = {\rm span}\{A'(p)\}.
$$
The first equation in (\ref{LS}) gives the relation
$$
v^{(m+1)}(p) = S v^{(m)}(p), \quad m \geq 1,
$$
which can be used to generate all eigenvectors from $v^{(1)}(p)$.
Using  summation formulae~\eqref{appendix6}-\eqref{appendix9} from the appendix
we have verified that the first eigenvector is given by
\begin{equation}\label{B1}
[v^{(1)}(p)]_n = \left\{
\begin{array}{ll}
p^2 \quad & \mbox{if} \quad n=0,\\
(1-p^2)[-(1+p^2)+n(1-p^2)]\,p^{n-2} \quad & \mbox{if} \quad n \in \mathbb{N}_+,\\
\end{array} \right.
\end{equation}
so that $S^* v^{(1)}(p) = (1-p^2) A'(p) \in {\rm ker}(L_+(p)) = {\rm ker}(2T(p)-M)$.
\end{proof}

Because of the commutation relation (\ref{commutation-tilde}), the operators $M^{-1} L_{\pm}(p)$ also have a common basis of
eigenvectors, which coincide with eigenvectors of the spectral problem (\ref{spectrum}) for nonzero eigenvalues $\Lambda$.
The following lemma fully characterizes eigenvalues of the spectral problem (\ref{spectrum}) in $\ell^2$.

\begin{lemma}
\label{lemma-L-3}
Eigenvalues of the spectral problem (\ref{spectrum}) are purely imaginary $\Lambda_m=\pm i \Omega_m$, where
\begin{equation}
\label{spectrum-L-plus-minus}
\Omega_0 = \Omega_1 = 0, \quad \Omega_m = \frac{m-1}{m+1}, \quad m \geq 2,
\end{equation}
independently of $p \in [0,1)$.
\end{lemma}

\begin{proof}
The spectral problem (\ref{spectrum}) can be written in the matrix form
\begin{equation}
\label{spectrum-matrix}
\mathcal{L}(p) \vec{ \texttt{a}} = \Lambda \mathcal{M} \vec{ \texttt{a}},
\end{equation}
where
$$
\mathcal{L}(p) = \left[ \begin{array}{cc} 0 & L_-(p) \\ -L_+(p) & 0 \end{array} \right], \quad
\mathcal{M} = \left[ \begin{array}{cc} M & 0 \\ 0 & M \end{array} \right], \quad
\vec{ \texttt{a}} = \left[ \begin{array}{c}  \texttt{a} \\  \texttt{b} \end{array} \right].
$$
The geometric kernel of $\mathcal{L}(p)$ is three-dimensional and spanned by the three
linearly independent eigenvectors
\begin{equation}
\label{eigenvectors-Zero}
\left[ \begin{array}{c}  0 \\  A(p) \end{array} \right], \quad
\left[ \begin{array}{c}  0 \\  M A(p) \end{array} \right], \quad
\left[ \begin{array}{c}  A'(p) \\ 0 \end{array} \right],
\end{equation}
according to (\ref{zeroLm}) and (\ref{zeroLp}). The generalized kernel of $\mathcal{M}^{-1} \mathcal{L}(p)$ is obtained
from solutions of the inhomogeneous equation $\mathcal{L}(p) \vec{\texttt{a}}_1 = \mathcal{M} \vec{\texttt{a}}_0$,
where $\vec{\texttt{a}}_0 \in {\rm ker}(\mathcal{L}(p))$. We have
\begin{equation}
\label{comp-2a}
\langle M A'(p), A(p) \rangle = 0, \quad \langle M A'(p), M A(p) \rangle \neq 0,
\end{equation}
thanks to equations (\ref{doublezeroLp}), (\ref{pos-eig-Lp}), and the explicit computation:
\begin{equation}
\label{comp-2}
\langle M A'(p), M A(p) \rangle = \sum_{n=0}^{\infty} (n+1)^2 p^{2n} \left[ n(1-p^2) - 2 p^2 \right] = \frac{2 p^2}{(1-p^2)^3} > 0.
\end{equation}
Therefore, the Jordan blocks are simple for the second and third eigenvectors in (\ref{eigenvectors-Zero})
and at least double for the first eigenvector in (\ref{eigenvectors-Zero}). Indeed,
the following generalized eigenvector follows from (\ref{doublezeroLp}):
\begin{equation}
\label{eigenvectors-Zero-gen}
\mathcal{L}(p) \vec{\texttt{a}}_1 = \mathcal{M} \vec{\texttt{a}}_0, \quad
\vec{\texttt{a}}_1 = -\frac{1}{2} \left[ \begin{array}{c}  A(p) \\ 0 \end{array} \right], \quad
\vec{\texttt{a}}_0 = \left[ \begin{array}{c}  0 \\  A(p) \end{array} \right].
\end{equation}
The corresponding Jordan block is exactly double because $M A(p)$ is not orthogonal to ${\rm ker}(L_-(p))$.
Hence the zero eigenvalue has multiplicity four with three eigenvectors in (\ref{eigenvectors-Zero}) and
one generalized eigenvector in (\ref{eigenvectors-Zero-gen}).

It remains to study the non-zero eigenvalues of the spectral problem (\ref{spectrum}).
To do so, we study negative eigenvalues of the operators $ M^{-1} L_{\pm}(p)$, which also commute.
Due to the presence of the operators $M$, we introduce a different complement of
the two-dimensional subspace $X_0(p) = {\rm span}\{ A(p), MA(p) \}$
compared to (\ref{X-0-constraints}). Namely, we define
\begin{equation}
\label{Y-0-constraints}
[X_c(p)]^{\perp} := \left\{ a \in \ell^2 : \quad \langle M A(p), a \rangle = \langle M^2 A(p), a \rangle = 0 \right\}.
\end{equation}
Eigenvectors for negative eigenvalues of $M^{-1} L_{\pm}(p)$ belong to $[X_c(p)]^{\perp}$.
Due to the first orthogonality condition in (\ref{Y-0-constraints}), we have
the same computation $B_{\pm}(p) a = 2 T(p) a$ as in (\ref{notations-T}).
Hence, the negative eigenvalues of $M^{-1} L_{\pm}(p)$ are identical to
the negative eigenvalues of $2 M^{-1} T(p) - I$ and hence they are identical to each other.
The spectral problem $L_{\pm} v = \lambda M v$ for $\lambda < 0$ can be rewritten in the equivalent form
\begin{equation}
\label{tilde-T}
T(p) v = \mu M v, \quad \mu := \frac{1+\lambda}{2}.
\end{equation}
Let us prove that the spectral problem (\ref{tilde-T}) admit a countable set of eigenvalues
\begin{equation}
\label{eigenvalues-mu}
\mu_m = \frac{1}{m+1}, \quad m \in \mathbb{N},
\end{equation}
with the corresponding eigenvectors given by
\begin{equation}
\label{eigenvectors-mu}
v^{(m)} = M^m A(p) - \sum_{j=0}^{m-1} \alpha_{j}^{(m)}(p) M^j A(p),
\end{equation}
where the coefficients $\{ \alpha^{(m)}_{j}(p) \}_{j=0}^{m-1}$ are uniquely
found from orthogonality conditions
$$
\langle M^j A(p), v^{(m)} \rangle = 0, \quad 0 \leq j \leq m-1.
$$
Consequently, $v^{(m)} \in [X_c(p)]^{\perp}$ for every $m \geq 2$.
Indeed, we have the first few eigenvalues and eigenvectors explicitly:
\begin{eqnarray*}
\mu_0 = 1 : & \quad & v^{(0)} = A(p), \\
\mu_1 = \frac{1}{2} : & \quad & v^{(1)} = M A(p) - \frac{1+p^2}{1-p^2} A(p), \\
\mu_2 = \frac{1}{3} : & \quad & v^{(2)} = M^2 A(p) + 3 \frac{1+p^2}{1-p^2} M A(p) - 2 \frac{1+p^2+p^4}{(1-p^2)^2} A(p),
\end{eqnarray*}
where we recognize the same eigenvectors $A(p)$ and $A'(p)$ for the first two (positive and zero) eigenvalues
of $M^{-1} L_{\pm}(p)$. Indeed,
$$
M^{-1} L_+(p) A(p) = 2 A(p), \quad M^{-1} L_+(p) A'(p) = 0,
$$
and
$$
M^{-1} L_-(p) A(p) = M^{-1} L_-(p) A'(p) = 0.
$$
In order to prove (\ref{eigenvalues-mu}) and (\ref{eigenvectors-mu}) for every $m \in \mathbb{N}$,
we represent
\begin{eqnarray}
\nonumber
[T(p) M^m A(p)]_n & = & (1-p^2) \left[ \sum_{j\in \mathbb{N}} (1+j)^m p^{|n-j|} p^j - \sum_{j\in \mathbb{N}} (1+j)^m p^{2+n+2j} \right]\\
\nonumber
& = & \left[ \sum_{j=0}^n (1+j)^m + \sum_{j=n+1}^{\infty} (1+j)^m p^{2(j-n)} - \sum_{j=0}^{\infty} (1+j)^m p^{2+2j} \right] A_n(p)\\
\label{powers-cancel}
& = & \left[ \sum_{k=1}^{n+1} k^m + p^2 \sum_{k=0}^{\infty} \left[ (1+k+1+n)^m - (1+k)^m \right] p^{2k} \right] A_n(p).
\end{eqnarray}
Because of the cancelation at the last sum at the zero power of $(n+1)$ in (\ref{powers-cancel}),
the right-hand side is written in positive powers of $(n+1)$ up to the $(m+1)$ power, namely,
\begin{eqnarray}
\label{decomposition-M-m}
T(p) M^m A(p) = \sum_{j=1}^{m+1} \beta_{j}^{(m)}(p) M^j A(p),
\end{eqnarray}
where $\{ \beta_{j}^{(m)}(p) \}_{j=1}^{m+1}$ are uniquely defined and $\beta_{0}^{(m)}(p) = 0$.
In particular, the first sum in (\ref{powers-cancel}) shows that $\beta_{m+1}^{(m)}(p) = \frac{1}{m+1}$, which yields the eigenvalue
$\mu_m = \frac{1}{m+1}$ of the spectral problem (\ref{tilde-T}). Now, adding a linear combination of
$m$ terms $\{ M^j A(p) \}_{j=0}^{m-1}$ to $M^m A(p)$ as in (\ref{eigenvectors-mu}) and using
expressions (\ref{decomposition-M-m}) inductively from $j = m-1$ to $j = 0$, we obtain a linear system of $m$
equations for $m$ coefficients $\{ \alpha_j^{(m)}(p) \}_{j=0}^{m-1}$. The linear system is associated with
a triangular matrix with nonzero diagonal coefficients, hence, it admits a unique solution for $\{ \alpha_j^{(m)}(p) \}_{j=0}^{m-1}$.
Hence, the validity of (\ref{eigenvalues-mu}) and (\ref{eigenvectors-mu}) is proven.

Thanks to the relation between $\mu$ and $\lambda$ in (\ref{tilde-T}), we have shown
that the negative eigenvalues of  $M^{-1} L_{\pm}(p)$ are given by
\begin{equation}
\label{lambda-m}
\lambda_m = 2\mu_m - 1 = -\frac{m-1}{m+1}, \quad m\geq 2.
\end{equation}
The common set of eigenvectors of $M^{-1} L_+(p)$ and $M^{-1} L_-(p)$ for negative eigenvalues $\lambda$ coincides with the set
of eigenvectors of the spectral problem (\ref{spectrum}) for nonzero eigenvalues $\Lambda$.
Thus, the nonzero eigenvalues of the spectral problem (\ref{spectrum})
are given by $\Lambda_m^2 = -\lambda_m^2$, which yields
the explicit expression (\ref{spectrum-L-plus-minus}) thanks to (\ref{lambda-m}).
\end{proof}

\begin{rem}
Similarly to the $N=0$ single-mode state (\ref{1mode}),
all eigenvalues of the spectral problem (\ref{spectrum}) for the ground state \eqref{ground-state} are semi-simple,
except for the double zero eigenvalue related to the gauge symmetry (\ref{symmshift1}).
\end{rem}

\section{Orbital stability of the $N = 0$ single-mode state}
\label{section-nonlinear}

Here we prove Theorem \ref{theorem-one-mode} which states the orbital stability of
the $N=0$ single-mode state (\ref{1mode}) with the normalization $c = 1$ or $\lambda = 1$.
Since this is the limit $p \to 0$ of the ground state \eqref{ground-state},
we will often use the expression $A_n(0) = \delta_{n0}$ as in (\ref{deltaKronecker}).
In order to prove Theorem \ref{theorem-one-mode}, we decompose a solution of the conformal flow (\ref{flow}) into a sum of the two-parameter
orbit of the ground state generated by the symmetries (\ref{symmscale}) and
(\ref{symmshift1}), as well as the symplectically orthogonal remainder term.
The following lemma provides a basis for such a decomposition.

\begin{lemma}
\label{lem-orthogonal}
There exist $\delta_0 > 0$
such that for every $\alpha \in \ell^2$ satisfying
\begin{equation}
\label{orth-given}
\delta := \inf_{\theta \in \mathbb{S}} \| \alpha -   e^{i \theta}  A(0) \|_{\ell^2} \leq \delta_0,
\end{equation}
there exists a unique choice of real-valued numbers $(c,\theta)$ and real-valued sequences $a,b \in \ell^2$
in the orthogonal decomposition
\begin{equation}
\label{orth-decomposition}
\alpha_n = e^{i \theta} \left( c A_n(0) + a_n + i b_n \right),
\end{equation}
subject to the orthogonality conditions
\begin{equation}
\langle MA(0), a \rangle = \langle MA(0), b \rangle = 0,
\end{equation}
satisfying the estimate
\begin{equation}
\label{orth-bound}
| c - 1 | + \| a + i b \|_{\ell^2} \lesssim \delta.
\end{equation}
\end{lemma}

\begin{proof}
The proof is based on the inverse function theorem applied to the vector function
$F(c,\theta;\alpha) : \mathbb{R}^2 \times \ell^2\mapsto \mathbb{R}^2$ given by
$$
F(c,\theta;\alpha) := \left[ \begin{array}{c} \langle MA(0), {\rm Re}(e^{-i \theta} \alpha - c  A(0)) \rangle \\
\langle MA(0), {\rm Im}(e^{-i \theta} \alpha - c  A(0)) \rangle \end{array} \right].
$$
The Jacobian matrix $DF$ at $A(0)$ is diagonal and invertible
$$
DF(1,0;A(0)) = -\left[ \begin{array}{cc} \langle MA(0), A(0) \rangle & 0 \\
0 & \langle MA(0), A(0) \rangle \end{array} \right].
$$
For sufficiently small $\delta > 0$, there exists a unique root
$(c,\theta)$ near $(1,\theta_0)$, where $\theta_0$ is an argument
in the infimum (\ref{orth-given}), with the bound
$$
|c-1| + |\theta - \theta_0| \lesssim \delta.
$$
 This proves the  bound  for $c$ in (\ref{orth-bound}). By using the definition of $(a,b)$ in the decomposition (\ref{orth-decomposition})
and the triangle inequality for $(c,\theta)$ near $(1,\theta_0)$,
it is then straightforward to show that $(a,b)$ are uniquely defined and satisfy the second bound in (\ref{orth-bound}).
\end{proof}

By Lemma \ref{lem-orthogonal},  any global solution $\alpha(t)\in h^1$ of the conformal flow system \eqref{flow}
satisfying for a sufficiently small positive $\epsilon$ and for every $t$,
\begin{equation}
\label{apriori-bound}
\inf_{\theta \in \mathbb{S}} \| \alpha(t) - e^{i\theta} A(0) \|_{\ell^2} \leq \epsilon
\end{equation}
admits  a unique decomposition in the form
\begin{equation}
\label{orth-decomposition-time}
\alpha_n(t) = e^{i \theta(t)} \left( c(t) A_n(0) + a_n(t) + i b_n(t) \right),
\end{equation}
where the remainder terms satisfy the symplectic orthogonality conditions
\begin{equation}
\label{projection1-2}
\langle MA(0), a(t) \rangle = \langle MA(0), b(t) \rangle = 0.
\end{equation}
Now we shall apply this decomposition to control the global solution starting from a small perturbation of the $N=0$ single-mode state.

\begin{lemma}
\label{lem-control-p-0}
Assume that initial data $\alpha(0) \in h^1$ satisfy
\begin{equation}\label{initial-data-p-0}
\| \alpha(0) - A(0) \|_{h^1} \leq \delta
\end{equation}
for some sufficiently small $\delta > 0$. Then, the corresponding unique global solution
$\alpha(t) \in C(\mathbb{R},h^1)$ of (\ref{flow})
can be represented by the decomposition (\ref{orth-decomposition-time})--(\ref{projection1-2})
satisfying for all $t$
\begin{equation}
\label{final-bound-p-0}
| c(t) - 1 | \lesssim \delta, \quad \| a(t) + i b(t) \|_{h^1} \lesssim \delta^{1/2}.
\end{equation}
\end{lemma}

\begin{proof}
Since $A_n(0) = \delta_{n0}$, the orthogonality conditions (\ref{projection1-2}) yield $a_0 = b_0 = 0$.
Substituting the representation (\ref{orth-decomposition-time}) into the conservation laws (\ref{charge}) and (\ref{lenergy}), we obtain
\begin{eqnarray}
\label{Q-expansion}
Q(\alpha(0)) = Q(\alpha(t)) = c(t)^2 + \sum_{n=1}^{\infty} (n+1) (a_n^2 + b_n^2)\,,
\end{eqnarray}
and
\begin{eqnarray}
\label{E-expansion}
E(\alpha(0)) = E(\alpha(t)) = c(t)^2 + \sum_{n=1}^{\infty} (n+1)^2 (a_n^2 + b_n^2)\,.
\end{eqnarray}
Thanks to (\ref{initial-data-p-0}) we have
\begin{equation}
\label{bound-Delta-p-0}
| Q(\alpha(0)) - 1| \lesssim \delta, \quad |E(\alpha(0)) - 1 | \lesssim \delta\,.
\end{equation}
Subtracting (\ref{Q-expansion}) from (\ref{E-expansion}) and using \eqref{bound-Delta-p-0}, we obtain
$$
\sum_{n=1}^{\infty} n (n+1) (a_n^2 + b_n^2) \lesssim \delta,
$$
which yields the second bound in (\ref{final-bound-p-0}). Substituting this bound into \eqref{E-expansion}
we obtain the first bound in (\ref{final-bound-p-0}). By continuity of the solution $\alpha(t) \in C(\mathbb{R},h^1)$
of (\ref{flow}), the bound (\ref{apriori-bound}) is satisfied for every $t \in \mathbb{R}$
if it is satisfied for $t = 0$. Therefore, by Lemma \ref{lem-orthogonal},
the decomposition (\ref{orth-decomposition-time})--(\ref{projection1-2}) holds for every $t$
and the bounds (\ref{final-bound-p-0}) are continued for every $t$.
\end{proof}

Theorem \ref{theorem-one-mode} is the  reinstatement of the result of Lemma \ref{lem-control-p-0}
with $\epsilon = \mathcal{O}(\delta^{1/2})$ or, equivalently, $\delta = \mathcal{O}(\epsilon^2)$.

\begin{rem}
Bounds in (\ref{final-bound-p-0}) can be improved for perturbations with $\alpha_0(0)=0$ since then $c(0) = 1$ in the decomposition
(\ref{orth-decomposition-time}). In this case, $\delta$ is replaced by $\delta^2$ in (\ref{bound-Delta-p-0}),
hence $\delta^{1/2}$ is replaced by $\delta$ in the bounds (\ref{final-bound-p-0}).
\end{rem}

\section{Orbital stability of the ground state}
\label{section-ground-state}

Here we prove Theorem \ref{theorem-ground-state}. As the first step
we establish some coercivity estimates for the operators $L_{\pm}(p)$
defined in \eqref{Hessian-r2} and \eqref{Hessian-explicit}. Let us redefine
$[X_c(p)]^{\perp}$ in (\ref{Y-0-constraints}) by using a different but equivalent choice
of the two orthogonality conditions:
\begin{equation}
\label{constraint-ground-state}
[X_c(p)]^{\perp} := \left\{ a \in \ell^2(\mathbb{N}) : \quad \langle MA(p), a \rangle = \langle M A'(p), a \rangle = 0 \right\}.
\end{equation}
This is a {\em symplectically orthogonal} subspace to $X_0(p) = {\rm span}\{ A(p), A'(p) \}$,
the two-dimensional subspace associated with the positive and zero eigenvalues of $L_{\pm}(p)$
 (recall (\ref{zeroLp}) that relates $A'(p)$ to $A(p)$ and $MA(p)$).
Let us introduce the symplectically orthogonal projection operator
$\Pi_c(p) : \ell^2(\mathbb{N}) \to [X_c(p)]^{\perp} \subset \ell^2(\mathbb{N})$.
The following lemma shows  that the operators $\Pi_c(p) L_{\pm}(p) \Pi_c(p)$ are negative and coercive
on $[X_c(p)]^{\perp}$ and that the coercivity constant is independent of $p \in [0,1)$.

\begin{lemma}
\label{lemma-L-4}
Given $a\in h^{1/2}$, for every $p \in [0,1)$ we have
\begin{equation}
\label{coercivity}
\langle \Pi_c(p) L_{\pm}(p) \Pi_c(p) a, a \rangle \lesssim -\| a \|^2_{h^{1/2}}.
\end{equation}
\end{lemma}

\begin{proof}
Recall that $L_+(p)$ has the simple positive eigenvalue $\lambda_*(p)$ with the eigenvector  $MA(p)$ and 
the simple zero eigenvalue with the eigenvector $A'(p)$, whereas 
the remaining eigenvalues are negative. Using \eqref{doublezeroLp}, we have
\begin{equation}
\label{comp-1}
\langle [L_+(p)]^{-1} M A(p), MA(p) \rangle = \frac{1}{2} \langle A(p), MA(p)\rangle > 0.
\end{equation}
By Theorem 4.1 in \cite{Pel-book}, this implies that the positive eigenvalue of $L_+(p)$ becomes a strictly negative eigenvalue of  $\Pi_c(p) L_+(p) \Pi_c(p)$
under the first constraint in (\ref{constraint-ground-state}).
On the other hand, since
\begin{equation}
\label{comp-1aa}
\langle M A'(p), A'(p) \rangle > 0,
\end{equation}
the zero eigenvalue of $L_+(p)$ becomes a strictly negative eigenvalue of $\Pi_c(p) L_+(p) \Pi_c(p)$ under
the second constraint in (\ref{constraint-ground-state}). Thus, $\Pi_c(p) L_+(p) \Pi_c(p)$
is strictly negative with the spectral gap (the distance between the negative spectrum of $\Pi_c(p) L_+(p) \Pi_c(p)$ and zero).
The coercivity bound (\ref{coercivity}) for $\Pi_c(p) L_+(p) \Pi_c(p)$
follows from standard spectral theorem and G{\aa}rding inequality
since the quadratic form for $\Pi_c(p) L_+(p) \Pi_c(p)$ is bounded in $h^{1/2}(\mathbb{N})$.

The operator $L_{-}(p)$ has a double zero eigenvalue and the remaining eigenvalues are negative. The eigenvectors
for the double zero eigenvalue coincide with $MA(p)$ and $A'(p)$, thanks to
 (\ref{zeroLp}) which relates $A'(p)$ to $A(p)$ and $MA(p)$.
The same argument as above yields the bound (\ref{coercivity}) for $\Pi_c(p) L_-(p) \Pi_c(p)$.
\end{proof}

In the second step we decompose  a solution of the system \eqref{flow} into a
 four-parameter family of ground states generated by
 the scaling (\ref{symmscale}) and gauge symmetries  (\ref{symmshift1}) and (\ref{symmshift2}),
 the  parameter $p \in [0,1)$,
as well as the symplectically orthogonal remainder term. More precisely, we have:

\begin{lemma}
\label{lem-orthogonal-p}
For every $p_0 \in (0,1)$, there exists $\delta_0 > 0$ such that
for every $\alpha \in \ell^2$ satisfying
\begin{equation}
\label{orth-given-p}
\delta := \inf_{\theta, \mu \in \mathbb{S}} \| \alpha - e^{i (\theta + \mu + \mu \cdot)} A(p_0) \|_{\ell^2} \leq \delta_0,
\end{equation}
there exists a unique choice for real-valued numbers $(c,p,\theta,\mu)$ and real-valued sequences $a,b \in \ell^2$
in the orthogonal decomposition
\begin{equation}
\label{orth-decomposition-p}
\alpha_n = e^{i (\theta + \mu + \mu n)} \left( c A_n(p) + a_n + i b_n \right),
\end{equation}
subject to the orthogonality conditions
\begin{equation}
\langle MA(p), a \rangle=
\langle MA'(p), a \rangle =
\langle MA(p), b \rangle =
\langle MA'(p), b \rangle =0,
\end{equation}
satisfying the estimate
\begin{equation}
\label{orth-bound-p}
| c - 1 | + |p - p_0 | + \| a + i b \|_{\ell^2} \lesssim \delta.
\end{equation}
\end{lemma}

\begin{proof}
The proof is based on the inverse function theorem applied to the vector function
$F(c,p,\theta,\mu;\alpha) : \mathbb{R}^4 \times \ell^2(\mathbb{N}) \mapsto \mathbb{R}^4$ given by
$$
F(c,p,\theta,\mu;\alpha) := \left[ \begin{array}{c} \langle MA(p), {\rm Re}(e^{-i (\theta + \mu + i \mu \cdot)} \alpha - cA(p)) \rangle \\
\langle MA'(p), {\rm Re}(e^{-i (\theta +i \mu + i \mu \cdot)} \alpha - cA(p)) \rangle \\
\langle MA(p), {\rm Im}(e^{-i (\theta +i \mu + i \mu \cdot)} \alpha - cA(p)) \rangle \\
\langle MA'(p), {\rm Im}(e^{-i (\theta +i \mu + i \mu \cdot)} \alpha - cA(p)) \rangle \end{array} \right],
$$
The Jacobian matrix $DF$ at $\alpha = A(p_0)$ is block-diagonal
$$
DF(1,p_0,0,0;A(p_0)) = \left[ \begin{array}{cc} D_1 & 0 \\ 0 & D_2 \end{array} \right],
$$
where
$$
D_1 = -\left[ \begin{array}{cc} \langle MA(p), A(p) \rangle & 0 \\ \langle MA'(p), A(p) \rangle & \langle MA'(p), A'(p) \rangle
\end{array} \right]
$$
and
$$
D_2 = -\left[ \begin{array}{cc}  \langle MA(p), A(p) \rangle & \langle MA(p), MA(p) \rangle \\
0 & \langle MA'(p), MA(p) \rangle \end{array} \right].
$$
Therefore, $DF(1,p_0,0,0;A(p_0))$ is invertible
with the $\mathcal{O}(1)$ bound on the inverse matrix for every $p_0 \in (0,1)$. Hence,
for sufficiently small $\delta > 0$, there exists a unique root
$(c,p,\theta,\mu)$ near $(1,p_0,\theta_0,\mu_0)$, where $(\theta_0,\mu_0)$ are
arguments in the infimum (\ref{orth-given-p}), with the bound
$$
|c-1| + |p-p_0| + |\theta - \theta_0| + |\mu - \mu_0| \lesssim \delta.
$$
Thus, the first two bounds in (\ref{orth-bound-p})
are satisfied for $c$ and $p$.
By using the definition of $(a,b)$ in the decomposition (\ref{orth-decomposition-p})
and the triangle inequality for $(c,p,\theta,\mu_0)$ near $(1,p_0,\theta_0,\mu_0)$,
it is then straightforward to show
that $(a,b)$ are uniquely defined and satisfies the last bound in (\ref{orth-bound-p}).
\end{proof}

For any global solution $\alpha(t) \in C(\mathbb{R},h^1)$  of \eqref{flow} which stays close to the ground state orbit
$\mathcal{A}(p_0)$ in $\ell^2$, i.e.
\begin{equation}
\label{apriori-bound-p}
\inf_{\theta, \mu \in \mathbb{S}} \| \alpha(t) - e^{i (\theta + \mu + \mu \cdot)} A(p_0) \|_{\ell^2}  \leq  \epsilon
\end{equation}
for a sufficiently small positive $\epsilon$,
Lemma \ref{lem-orthogonal-p} yields the unique decomposition in the form
\begin{equation}
\label{orth-decomposition-time-p}
\alpha_n(t) = e^{i (\theta(t) +  (n+1) \mu(t))} \left( c(t) A_n(p(t)) + a_n(t) + i b_n(t) \right),
\end{equation}
where the remainder terms satisfy the symplectic orthogonality conditions
\begin{equation}
\label{projection1}
\langle MA(p(t)), a(t) \rangle =
\langle MA'(p(t)), a(t) \rangle = 0
\end{equation}
and
\begin{equation}
\label{projection2}
\langle MA(p(t)), b(t) \rangle
= \langle MA'(p(t)), b(t) \rangle =0.
\end{equation}

By the coercivity bounds in Lemma \ref{lemma-L-4}, we control $c(t)$, $a(t)$, and $b(t)$ as follows.

\begin{lemma}
\label{lem-control-p}
Assume that the initial data $\alpha(0) \in h^1$
satisfy
\begin{equation}
\label{initial-data-p}
\| \alpha(0) - A(p_0) \|_{h^1} \leq \delta,
\end{equation}
for some $p_0 \in [0,1)$ and a sufficiently small $\delta > 0$.
Then, the corresponding unique global solution
$\alpha(t) \in C(\mathbb{R}_+,h^1)$ of (\ref{flow})
can be represented by the decomposition (\ref{orth-decomposition-time-p}) with
(\ref{projection1}) and (\ref{projection2}) satisfying for all $t$
\begin{equation}
\label{final-bound-p}
| c(t) - 1 | + \| a(t) + i b(t) \|_{h^{1/2}} \lesssim \delta.
\end{equation}
\end{lemma}

\begin{proof}
Let us define the  function
$$
\Delta(c) := c^2 \left( Q(\alpha) - 1 \right) - \frac{1}{2} \left( H(\alpha) - 1 \right).
$$
Evaluating this function for the decomposition (\ref{orth-decomposition-time-p}) with (\ref{projection1})
and (\ref{projection2}), and  using the variational characterization of the
ground state \eqref{K}, we obtain
\begin{eqnarray}
\Delta(c(t)) & = & c^2(t) \left( Q(\alpha(t)) - 1 \right) - \frac{1}{2} \left( H(\alpha(t)) - 1 \right) \nonumber \\
& = & \frac{1}{2} \left( c(t)^2 - 1 \right)^2 - c(t)^2 \langle L_+(p) a(t),a(t) \rangle
- c(t)^2 \langle L_-(p) b(t),b(t) \rangle \nonumber \\
& \phantom{t} & \phantom{text} + N(a(t),b(t),c(t)), \label{Delta-c1}
\end{eqnarray}
where the cubic and quartic terms in $N$ satisfy the bound
\begin{equation}
\label{bound-N}
|N(a,b,c)| \lesssim   |c| \| a \|^3_{h^{1/2}} + |c| \| b \|^3_{h^{1/2}}
+ \| a \|^4_{h^{1/2}} + \| b \|^4_{h^{1/2}}.
\end{equation}
On the other hand,
we have
\begin{equation}\label{Delta-c}
\Delta(c) = \Delta(1) + (c^2-1) \left( Q(\alpha(t)) - 1\right),
\end{equation}
where, thanks to (\ref{initial-data-p}) and the conservation of $Q(\alpha(t))$ and $H(\alpha(t))$, we have
\begin{equation}
\label{bound-Delta}
|\Delta(1) | \lesssim \delta^2,\quad |Q(\alpha(t)) - 1| \lesssim \delta.
\end{equation}
Combining \eqref{Delta-c1} with \eqref{Delta-c}, using the bounds \eqref{bound-N} and \eqref{bound-Delta}, as well as
the coercivity bounds (\ref{coercivity}) in Lemma \ref{lemma-L-4},
we obtain
\begin{eqnarray*}
\left( c(t)^2 - 1 \right)^2 +  \| a(t)  \|^2_{h^{1/2}} +  \| b(t) \|^2_{h^{1/2}} \lesssim \delta^2,
\end{eqnarray*}
which  yields the proof of the bound (\ref{final-bound-p}).
\end{proof}

The bound (\ref{final-bound-p}) controls the drift of the parameter $c(t)$ and the
perturbation terms $a(t)$ and $b(t)$ in the $h^{1/2}$ norm.
The following lemma uses the conservation of $E$ to control
 $a(t)$ and $b(t)$  in the $h^1$ norm
modulo a possible drift of the parameter $p(t)$ towards smaller values.

\begin{lemma}
\label{lem-control-extended}
Under the same assumption (\ref{initial-data-p}) for the initial data $\alpha(0)$ as in Lemma~\ref{lem-control-p},
the corresponding unique global solution
$\alpha(t) \in C(\mathbb{R}_+,h^1)$ of (\ref{flow})
can be represented by the decomposition (\ref{orth-decomposition-time-p}) with
(\ref{projection1}) and (\ref{projection2}) satisfying for all $t$
\begin{equation}
\label{final-bound-extended}
-p_0 \leq p(t) - p_0 \lesssim\delta, \quad \| a(t) + i b(t) \|_{h^1} \lesssim  \delta^{1/2}
+ |p_0 - p(t)|^{1/2}.
\end{equation}
\end{lemma}

\begin{proof}
Substituting the decomposition (\ref{orth-decomposition-time-p}) into
the expression for $E(\alpha)$ and using the orthogonality conditions
(\ref{projection1}) and (\ref{projection2}), we obtain
\begin{eqnarray}
\label{E-expansion-p}
E(\alpha(t)) = c(t)^2 \frac{1 + p(t)^2}{1 - p(t)^2} + \langle M a(t), M a(t) \rangle + \langle M b(t), M b(t) \rangle.
\end{eqnarray}
On the other hand, thanks to \eqref{initial-data-p}  and the conservation of $E(\alpha(t))$, we have
$E(\alpha(t)) = E(\alpha(0))$ with
\begin{equation}\label{bound-Delta-p}
\left| E(\alpha(0)) - \frac{1+p_0^2}{1-p_0^2} \right| \lesssim \delta.
\end{equation}
Combining \eqref{E-expansion-p} and \eqref{bound-Delta-p} and using
the bound (\ref{final-bound-p}) with $c(t)=1+\mathcal{O}(\delta)$, we obtain
\begin{equation}
\label{expansion-E}
 \frac{2(p(t)^2-p_0^2)}{(1-p(t)^2)(1-p_0^2)} + \| a(t) + i b(t) \|_{h^1}^2 \lesssim \delta,
\end{equation}
which yields the bound (\ref{final-bound-extended}).
\end{proof}

Theorem \ref{theorem-ground-state} is a reinstatement of the results of Lemmas \ref{lem-control-p} and \ref{lem-control-extended}
with $\epsilon = \mathcal{O}(\delta^{1/2})$ or, equivalently, $\delta = \mathcal{O}(\epsilon^2)$.
If $p(t) \in [0,p_0]$ drifts away from $p_0$, then the assumption (\ref{apriori-bound-p})
will be violated for some $t_0>0$. However, the decomposition (\ref{orth-decomposition-time-p})
with (\ref{projection1}) and (\ref{projection2}) is justified by continuation
of the solution $\alpha(t) \in C(\mathbb{R}_+,h^1)$ in a local neighborhood of
the ground state orbit $\mathcal{A}(p(t_0))$ for every $t_0$.

As mentioned above, Theorem \ref{theorem-ground-state} does not exclude
a drift in time of the parameter $p$ towards smaller values. To see if this is only a technical difficulty,
we try here a completely different approach, based on compactness-type arguments.  In doing so,
we follow very  closely the proof of Corollary~4 in \cite{GG10} which states the orbital stability
of the ground state for the cubic Szeg\H{o} equation.

From Theorem~2 (which is the analog of Proposition~5 in \cite{GG10}) we know that
the functional $G=Q^2-H$ attains  the global minimum equal to zero on the family of ground states.
Let $\alpha^{(j)}(0)$ be a sequence of initial data in $h^1$ such that
  \begin{equation}\label{minseq}
    \|\alpha^{(j)}(0)-A(p_0)\|_{h^1} \rightarrow 0 \quad \mbox{\rm as} \quad j \to \infty,
   \end{equation}
   for some $p_0\in [0,1)$
    and let $\alpha^{(j)}(t)$ be the corresponding sequence of solutions. By global existence in Theorem \ref{theorem-evolution}
    and conservation of $E(\alpha) = \| \alpha \|^2_{h^1}$, the sequence $(\alpha^{(j)}(t))_{j \in \mathbb{N}}$
    is bounded for every fixed $t$. Therefore, it has a subsequence, which
    converges weakly to some $\alpha(t) \in h^1$ for every $t$.

Since the functionals $Q(\alpha)$ and $H(\alpha)$ are continuous in the weak topology of $h^1$, it follows that
  \begin{equation}\label{qh0}
    Q(\alpha^{(j)}(0)) \rightarrow 1,\quad H(\alpha^{(j)}(0)) \rightarrow 1,\quad G(\alpha^{(j)}(0)) \rightarrow 0,
  \end{equation}
  hence by conservation laws
\begin{equation}\label{qh}
   Q(\alpha^{(j)}(t)) \rightarrow 1,\quad H(\alpha^{(j)}(t)) \rightarrow 1,\quad G(\alpha^{(j)}(t))\rightarrow 0.
  \end{equation}
  By weak continuity of $Q$ and $H$, we have
  \begin{equation}\label{qh}
   Q(\alpha(t)) = 1,\quad H(\alpha(t)) = 1,\quad G(\alpha(t))= 0.
  \end{equation}
From the  inequality $G\geq 0$ and the uniqueness of the global minimizer (modulo gauge symmetries)
we conclude that $\alpha^{(j)}(t)$ converges strongly in $h^1$ to $\alpha(t) \in \mathcal{A}(p(t))$ for every
$t$. This is equivalent to the bound (\ref{final-bound-p}) in Lemma \ref{lem-control-p}.

From the lower semi-continuity  of the conserved quantity $E$ we get
\begin{equation}\label{E-bound}
  E(\alpha(t))=\frac{1+p(t)^2}{1-p(t)^2} \leq \lim_{j\rightarrow\infty} E(\alpha^{(j)}(0))=\frac{1+p_0^2}{1-p_0^2}
\end{equation}
which eliminates the drift towards $p(t) > p_0$ but not towards $p(t) < p_0$. This is equivalent to
the bound (\ref{final-bound-extended}) in Lemma \ref{lem-control-extended}.

In view of the fact that two different approaches leave open a possibility of drift of $p(t)$ to lower values, it would be interesting to explore this issue numerically.

\vspace{0.2cm}
\noindent {\bf Acknowledgments:} We thank  M.~Maliborski for numerical tests of orbital stability. This research was supported  by the Polish National Science Centre grant no.\ DEC-2012/06/A/ST2/00397.
PB acknowledges  CUniverse
research promotion project by Chulalongkorn University (grant reference CUAASC) and  O.~Evnin  for helpful discussions.
\appendix
\section{Some useful identities}

Let us record identities for partial sums of geometric series
\begin{align}
& \sum_{k=0}^n p^{2k} = \frac{1-p^{2n+2}}{1-p^2}, \label{appendix7} \\
& \sum_{k=0}^n k p^{2k} = \frac{p^2 (1-(n+1)p^{2n} + n p^{2n+2}}{(1-p^2)^2}.\label{appendix8}
\end{align}
In order to verify the expression \eqref{Hessian-explicit}, we have used the following computations:
\begin{align}
j \leq n : & \quad
\sum_{k=0}^{\infty} S_{njk,n+k-j} p^{n+2k-j} = \sum_{k=0}^{j}(k+1) p^{n+2k-j} + \sum_{k=j+1}^{\infty}(j+1) p^{n+2k-j} \nonumber \\
& \quad \quad \quad = \frac{1}{(1-p^2)^2} (p^{n-j}-p^{2+j+n}), \label{appendix1}\\
j \geq n : & \quad
\sum_{k=j-n}^{\infty} S_{njk,n+k-j} p^{n+2k-j} = \sum_{k=j-n}^{j}(n+k-j+1) p^{n+2k-j} + \sum_{k=j+1}^{\infty}(n+1) p^{n+2k-j} \nonumber \\
& \quad \quad \quad  = \frac{1}{(1-p^2)^2} (p^{j-n}-p^{2+j+n}), \label{appendix2}\\
j \geq n : & \quad
\sum_{k=0}^{n+j}S_{njk,n+j-k} = \sum_{k=0}^{n}(k+1) + \sum_{k=n+1}^{j}(n+1) + \sum_{k=j+1}^{j+n}(n+j-k+1) \nonumber \\
& \quad \quad \quad = (1+j)(1+n), \label{appendix3}
\end{align}
where the identities (\ref{appendix7})--(\ref{appendix8}) have been used.

In order to verify the expression \eqref{B1} for the eigenvector to the eigenvalue $-1$
of the operators $L_{\pm}(p)$, or equivalently, of the operator $2T(p) - M$, we write explicitly
\begin{eqnarray*}
\left[ (2T(p)-M) v^{(1)} \right]_n & = & p^2 \left[ 2p^n - \delta_{n0} \right]
+ (1-p^2) \sum_{k=1}^{\infty} \left[ 2p^{|n-k|} - (n+1) \delta_{nk} \right] \\
& \phantom{t} & \times \left[ k(1-p^2) - (1+p^2) \right] p^{k-2}.
\end{eqnarray*}
By using the following identities
\begin{align}
& \sum_{k=1}^{\infty}p^{k+|n-k|} = \frac{p^2+n(1-p^2)}{(1-p^2)} p^n, \label{appendix6} \\
& \sum_{k=1}^{\infty} k p^{k+|n-k|} = \frac{2p^2+n(1-p^4)+n^2(1-p^2)^2}{2 (1-p^2)^2} p^n, \label{appendix9}
\end{align}
we have verified that $L_{\pm} v^{(1)} = (2T(p)-M) v^{(1)} = -v^{(1)}$.

\end{document}